\batchmode
\makeatletter
\def\input@path{{"/Users/russw/Documents/Research/mypapers/A modular characterization of supersolvable lattices/"}}
\makeatother
\documentclass[12pt,oneside,english,lowtilde]{amsart}
\usepackage[T1]{fontenc}
\usepackage[latin9]{inputenc}
\usepackage{geometry}
\usepackage{babel}
\usepackage{wrapfig}
\usepackage{url}
\usepackage{enumitem}
\usepackage{amstext}
\usepackage{amsthm}
\usepackage{amssymb}
\usepackage{graphicx}
\usepackage[dvips,unicode=true,pdfusetitle,
 bookmarks=true,bookmarksnumbered=false,bookmarksopen=false,
 breaklinks=false,pdfborder={0 0 1},backref=false,colorlinks=false]
 {hyperref}
\usepackage{breakurl}

\makeatletter
\numberwithin{equation}{section}
\numberwithin{figure}{section}
\theoremstyle{plain}
\newtheorem{thm}{\protect\theoremname}[section]
\theoremstyle{plain}
\newtheorem{lem}[thm]{\protect\lemmaname}
\theoremstyle{remark}
\newtheorem{rem}[thm]{\protect\remarkname}
\theoremstyle{plain}
\newtheorem{question}[thm]{\protect\questionname}

\usepackage{lmodern}

\makeatother

\providecommand{\lemmaname}{Lemma}
\providecommand{\questionname}{Question}
\providecommand{\remarkname}{Remark}
\providecommand{\theoremname}{Theorem}

\begin{document}
\global\long\def\cc{\mathbb{C}}%

\global\long\def\link{\operatorname{link}}%

\title{A modular characterization of supersolvable lattices}
\author{Stephan Foldes and Russ Woodroofe}
\thanks{Work of the second author is supported in part by the Slovenian Research
Agency (research program P1-0285 and research projects J1-9108, N1-0160,
J1-2451).}
\address{Miskolci Egyetem, 3515 Miskolc-Egyetemvaros, Hungary}
\email{foldes.istvan@uni-miskolc.hu}
\address{Univerza na Primorskem, Glagoljaška 8, 6000 Koper, Slovenia}
\email{russ.woodroofe@famnit.upr.si}
\urladdr{\url{https://osebje.famnit.upr.si/~russ.woodroofe/}}
\begin{abstract}
We characterize supersolvable lattices in terms of a certain modular
type relation. McNamara and Thomas earlier characterized this class
of lattices as those graded lattices having a maximal chain that consists
of left-modular elements. Our characterization replaces the condition
of gradedness with a second modularity condition on the maximal chain
of left-modular elements.
\end{abstract}

\maketitle

\section{\label{sec:Introduction}Introduction and background}

A supersolvable lattice is one that is well-behaved in a similar way
to the subgroup lattice of a supersolvable group. Stanley in a 1972
paper \cite{Stanley:1972} defined a lattice $L$ to be \emph{supersolvable}
if $L$ admits a maximal chain $\mathbf{m}$, which we call a \emph{chief
chain} (or \emph{$M$-chain}), so that the sublattice generated by
$\mathbf{m}$ and any other chain $\mathbf{c}$ is distributive. Since
then, this class of lattices has been characterized in a number of
different ways. We first give the list of characterizations, and will
explain the terminology used shortly after.
\begin{thm}[Liu, McNamara and Thomas]
\label{thm:OldChars} For a finite lattice $L$ and maximal chain
$\mathbf{m}$, the following are equivalent.
\begin{enumerate}
\item \label{enu:Supersolvable}$\mathbf{m}$ is a chief chain (and so $L$
is supersolvable),
\item $L$ admits an $EL$-labeling with ascending chain $\mathbf{m}$ so
that every maximal chain is labeled with a permutation of a fixed
label set $[n]$, and
\item \label{enu:Char-GradedLM}$L$ is graded and $\mathbf{m}$ consists
of left-modular elements.
\end{enumerate}
\end{thm}

The direction (3) $\implies$ (2) was first proved in Liu's PhD thesis
\cite{Liu:1999}; see also \cite{Liu/Sagan:2000}. McNamara showed
(1) $\iff$ (2) as part of his own PhD work, and published the results
in \cite{McNamara:2003}. McNamara and Thomas noticed the connection
with Liu's work in \cite{McNamara/Thomas:2006}, completing Theorem~\ref{thm:OldChars}.

Stanley's paper \cite{Stanley:1972} on supersolvable lattices has
been cited hundreds of times. Supersolvability is useful for understanding
examples such as non-crossing partition lattices \cite{Hersh:1999,Hersh:2003}.
Recent papers that discuss supersolvable lattices include \cite{Adaricheva:2017,Can:2012,Hallam/Sagan:2015,Muhle:2018}.

We now explain the definitions used in Theorem~\ref{thm:OldChars}.
A \emph{edge labeling over $\mathbb{Z}$} of a lattice assigns an
integer label to each edge of the Hasse diagram. Thus, we associate
a word in the label set to each maximal chain on an interval. An \emph{$EL$-labeling}
is an edge labeling over $\mathbb{Z}$ so that each interval has a
unique lexicographically earliest maximal chain, and so that this
chain is the only ascending maximal chain on the interval. We will
not discuss $EL$-labelings further here, although closely-related
notions were a main motivation for the definition of supersolvable
lattices \cite{Bjorner:1980,Stanley:1972,Stanley:1974}.

Modularity requires more discussion. A lattice is \emph{modular} if
for every three elements $z$ and $x<y$, we have $\left(x\vee z\right)\wedge y=x\vee\left(z\wedge y\right)$.
Equivalent to this condition, though not immediately so, is that for
every such $z$ and $x<y$, we have either $z\wedge x\neq z\wedge y$
or $z\vee x\neq z\vee y$. Thus, the sublattice generated by $x$,
$y$, and $z$ is not a pentagon, so we see that a modular lattice
is exactly one with no pentagon sublattice.

One may think of a modular lattice as being analogous to the subgroup
lattice of an abelian group. There are multiple conditions on lattice
elements that are analogous to normality of a subgroup in the subgroup
lattice of a finite group, most or all of which come down to requiring
the modular identity $x\vee\left(z\wedge y\right)=\left(x\vee z\right)\wedge y$
for certain elements $z$ and $x<y$. A pair of lattice elements $(z,y)$
is a \emph{modular pair} if whenever $x<y$, we have $x\vee\left(z\wedge y\right)=\left(x\vee z\right)\wedge y$.
Now we say that an element $m$ of the lattice $L$ is \emph{left-modular}
if $\left(m,y\right)$ forms a modular pair for each $y\in L$, and
that $m$ is \emph{right-modular} if $\left(z,m\right)$ is a modular
pair for each $z\in L$. If $m$ satisfies both, we say it is \emph{two-sided
modular}.

\begin{wrapfigure}{r}{0.37\columnwidth}%
\begin{centering}
\includegraphics{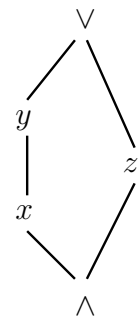}
\par\end{centering}
\centering{}\caption{$\ $A non-modular lattice}
\end{wrapfigure}%

Dedekind's modular identity from group theory says that if $A$ and
$B\subseteq C$ are subgroups of a group $G$, then $B(A\cap C)=BA\cap C$.
It follows easily from Dedekind's modular identity that a normal subgroup
in the dual of a subgroup lattice is two-sided modular. (It may be
helpful to recall here that if $A$ normalizes $B$, then $BA=AB=A\vee B$.)
Stanley showed in \cite{Stanley:1972} that if $L$ admits a maximal
chain of two-sided modular elements, then $L$ is supersolvable. It
is easy to find counterexamples to the converse, as supersolvability
is a self-dual property, but right modularity is not. Theorem~\ref{thm:OldChars}~(\ref{enu:Char-GradedLM})
is a partial converse to Stanley's condition, but requires the global
lattice property of gradedness.

In this paper, we add two conditions to the list of characterizations
of supersolvable lattices. First, we give a characterization of supersolvability
purely in terms of modular conditions on elements. We say that a chain
of elements $\mathbf{m}$ in lattice $L$ is \emph{right chain-modular}
if for every pair $x<y$ of elements in $\mathbf{m}$ and every $z\in L$,
we have $x\vee\left(z\wedge y\right)=\left(x\vee z\right)\wedge y$.
For convenience, we say that the chain $\mathbf{m}$ is \emph{chain-modular}
if it is right chain-modular and if every element in the chain is
left-modular.
\begin{thm}
\label{thm:MainTheorem}The maximal chain $\mathbf{m}$ of finite
lattice $L$ is a chief chain if and only if
\begin{enumerate}[label=(4),ref=4]
\item \label{enu:ChainModular}$\mathbf{m}$ is a chain-modular maximal
chain.
\end{enumerate}
\end{thm}

We will give two different proofs of Theorem~\ref{thm:MainTheorem},
both showing directly that condition (\ref{enu:ChainModular}) implies
condition (\ref{enu:Supersolvable}), and also that it implies condition
(\ref{enu:Char-GradedLM}). A consequence will be a new proof of the
equivalence of these two conditions that avoids the theory of $EL$-labelings.
Thomas previously gave such a proof in \cite{Thomas:2005} (see also
\cite{Thomas:2006}), but this proof required the theory of free lattices.
Our proofs will be more elementary.

An additional result of this paper will be useful as an intermediate
step in some of our proofs, and extends conditions well-known in the
special case of geometric lattices (see \cite[Corollary 2.3]{Stanley:1972}
and \cite[Theorem 3.3]{Brylawski:1975}, also \cite{Birkhoff:1967}).
We say that an element $m$ of the graded lattice $L$ with rank function
$\rho$ is \emph{rank modular} if for every $x$ in $L$ the identity
$\rho(m\vee x)+\rho(m\wedge x)=\rho(m)+\rho(x)$ holds. We will show:
\begin{lem}
\label{lem:RankEqLeftModular}Let $L$ be a graded finite lattice
with rank function $\rho$. The element $m$ of $L$ is left-modular
if and only if $m$ is rank modular.
\end{lem}

We say that a maximal chain is \emph{rank modular} if it consists
of rank modular elements. In view of Theorem~\ref{thm:OldChars}~(3),
an immediate consequence of Lemma~\ref{lem:RankEqLeftModular} will
be:
\begin{thm}
\label{thm:RankMod} The maximal chain $\mathbf{m}$ of finite lattice
$L$ is a chief chain if and only if
\begin{enumerate}[label=(5),ref=5]
\item \label{enu:RankModular}$L$ is graded, and $\mathbf{m}$ is a rank
modular maximal chain.
\end{enumerate}
\end{thm}

\begin{rem}
The role of a chief chain in a supersolvable lattice generalizes that
of the chief series in the subgroup lattice of a supersolvable group.
This explains our terminology. Stanley \cite{Stanley:1972,Stanley:2012}
and other authors call this chain an \emph{$M$-chain}. Theorem~\ref{thm:MainTheorem}
explains exactly the extent to which this $M$ stands for ``modular.''
\end{rem}

A basic example of a supersolvable lattice is that of the partition
lattice $\Pi_{n}$, consisting of all partitions of $[n]$, ordered
by refinement. The elements in $\Pi_{n}$ having at most one non-singleton
block are easily seen to coincide with those satisfying the rank modular
condition. Björner \cite{Bjorner:1987} and Haiman \cite{Haiman:1994}
extended the family of finite partition lattices to a non-discrete
limit lattice. They observed that a similar condition to that of Theorem~\ref{thm:RankMod}
holds in this setting with respect to a generalized rank function.
It would be interesting to extend other characterizations of supersolvable
lattices to the non-discrete setting.

\subsection*{Organization}

This paper is organized as follows. In Section~\ref{sec:Background}
we lay out a lemma and other preliminary material for use throughout.
In Section~\ref{sec:StanleySupersolvable} we prove that Condition~(\ref{enu:ChainModular})
is equivalent to Condition~(\ref{enu:Supersolvable}), which gives
one proof of Theorem~\ref{thm:MainTheorem}. In Section~\ref{sec:GradedLM}
we first show that Conditions~(\ref{enu:Char-GradedLM}) and (\ref{enu:RankModular})
are equivalent, then give a proof, independent of Section~\ref{sec:StanleySupersolvable},
that both are equivalent to (\ref{enu:ChainModular}). This gives
a second proof of Theorem~\ref{thm:MainTheorem}. When combined with
Section~\ref{sec:StanleySupersolvable}, the results of this section
give a new simple proof that Condition~(\ref{enu:Char-GradedLM})
implies Condition~(\ref{enu:Supersolvable}). In Section~\ref{sec:CheckingChainMod},
we give a simple condition for checking chain-modularity of a maximal
chain. We close in Section~\ref{sec:Questions} by asking about generalizations
to non-discrete lattices.

\subsection*{Acknowledgements}

We thank the anonymous referee for thoughtful and helpful comments.

\section{\label{sec:Background}Preliminaries}

In this section, we give some basic results on modular elements in
a lattice. Additional background can be found in textbooks such as
\cite[Chapter 3]{Stanley:2012} or \cite[Chapter V]{Gratzer:2011a}.
We generally follow the notation of \cite[Chapter 3]{Stanley:2012}.

First, in any lattice $L$ and for any $x<y$ and $z$, we have the
\emph{modular inequality} 
\[
x\vee(z\wedge y)\leq(x\vee z)\wedge y.
\]
Thus, the modular conditions on elements that we have discussed give
sufficient conditions for the modular inequality to be an equality.

\begin{figure}
\begin{centering}
\includegraphics{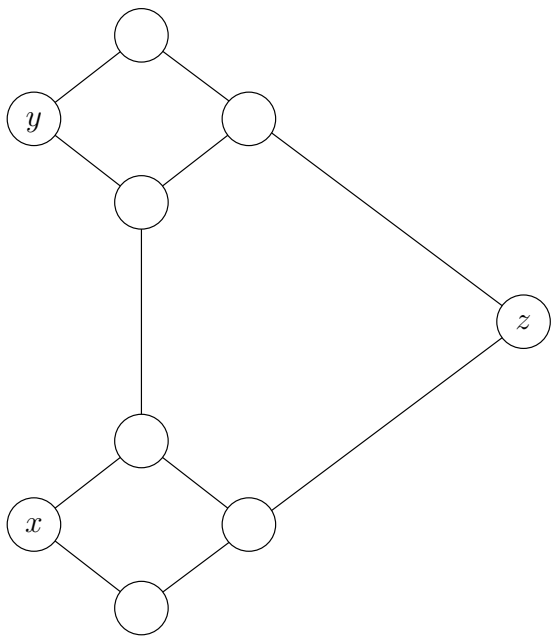}
\par\end{centering}
\caption{\label{fig:ModularGeneral}A strict modular inequality on $x<y$ and
$z$ yields a pentagon sublattice.}
\end{figure}

The general form of a lattice generated by $x<y$ and $z$ is shown
in Figure~\ref{fig:ModularGeneral} (see \cite[Figure 6]{Gratzer:2011a}).
It is easy to see from this figure that an element is left-modular
if and only if it avoids being the short side of a pentagon sublattice,
but that the situation with right-modularity or right chain-modularity
is somewhat more complicated. Here, by the \emph{short side} of the
pentagon sublattice, we mean the element $z$ of Figure~\ref{fig:ModularGeneral};
similarly, by the \emph{long side} we mean the elements $x<y$.
\begin{lem}[{Pentagonal characterization of left-modularity \cite[Theorem 1.4]{Liu/Sagan:2000}}]
 \label{lem:PentagonsLM} An element $z$ of a lattice $L$ is left-modular
if and only if for each $x<y$ either $x\vee z\neq y\vee z$ or $x\wedge z\neq y\wedge z$.

If $z$ is left-modular and in addition $x\lessdot y$, then exactly
one of the equalities $x\vee z=y\vee z$ and $x\wedge z=y\wedge z$
holds.
\end{lem}

\section{\label{sec:StanleySupersolvable}Equivalence with supersolvability}

In this section, we show the equivalence of Conditions (\ref{enu:Supersolvable})
and (\ref{enu:ChainModular}). That is, we show that a maximal chain
$\mathbf{m}$ is a chief chain if and only if it is chain-modular.

If $\mathbf{m}$ is a chief chain, then the sublattice generated by
$\mathbf{m}$ and any other chain is distributive. Since distributive
lattices have no pentagon sublattices, it follows immediately by considering
$1$-element chains that $\mathbf{m}$ is right chain-modular, and
by considering $2$-element chains that $\mathbf{m}$ consists of
left-modular elements.

Conversely, if $\mathbf{m}$ is a right chain-modular maximal chain
then, following Stanley \cite[Proposition 2.1]{Stanley:1972} and
Birkhoff \cite[pp. 65-66]{Birkhoff:1967}, it suffices to prove a
pair of dual identities, as follows. For the reader's convenience,
we use notation similar to that in \cite{Stanley:1972}.
\begin{lem}
\label{lem:ExtBirkhoffLemma}Let $L$ be a lattice and $\mathbf{m}$
be a chain-modular chain in $L$. If $a_{1}\geq\cdots\geq a_{r}$
are elements from $\mathbf{m}$ and $b_{1}\leq\cdots\leq b_{r}$ are
elements from an arbitrary chain in $L$, then 
\begin{align}
(b_{1}\vee a_{1})\wedge\cdots\wedge(b_{r}\vee a_{r}) & =b_{1}\vee(a_{1}\wedge b_{2})\vee\cdots\vee(a_{r-1}\wedge b_{r})\vee a_{r},\text{ and}\label{eq:BirkhoffCoatoms}\\
(a_{1}\wedge b_{1})\vee\cdots\vee(a_{r}\wedge b_{r}) & =a_{1}\wedge(b_{1}\vee a_{2})\wedge\cdots\wedge(b_{r-1}\vee a_{r})\wedge b_{r}.\label{eq:BirkhoffAtoms}
\end{align}
\end{lem}

\begin{proof}
We proceed by simultaneous induction on $r$. We may assume without
loss of generality that $\hat{0},\hat{1}\in\mathbf{m}$.

By duality, it suffices to prove the first identity, so we take $x$
to be the left-hand side element of (\ref{eq:BirkhoffCoatoms}). We
apply several modular identities. First, since $b_{1}\leq(b_{2}\vee a_{2})\wedge\cdots(b_{r}\wedge a_{r})$,
by left-modularity of $a_{1}$ we can rewrite 
\begin{align*}
x & =b_{1}\vee\big(a_{1}\wedge(b_{2}\vee a_{2})\wedge\cdots\wedge(b_{r}\vee a_{r})\big),\\
 & =b_{1}\vee\bigg(a_{1}\wedge\big(b_{2}\vee(a_{2}\wedge b_{3})\vee\cdots(a_{r-1}\wedge b_{r})\vee a_{r}\big)\bigg),
\end{align*}
where the latter equality is by induction. Now, as $a_{r}\leq a_{1}$
are in $\mathbf{m}$, we can apply right chain-modularity to rewrite
\[
x=b_{1}\vee\bigg(a_{1}\wedge\big(b_{2}\vee(a_{2}\wedge b_{3})\vee\cdots(a_{r-1}\wedge b_{r})\big)\bigg)\vee a_{r}.
\]
At this point, we notice that $b_{2}=\hat{1}\wedge b_{2}$, so we
can apply (\ref{eq:BirkhoffAtoms}) inductively to rewrite 
\begin{align*}
x & =b_{1}\vee\bigg(a_{1}\wedge\big((b_{2}\vee a_{2})\wedge\cdots\wedge(b_{r-1}\vee a_{r-1})\wedge b_{r}\big)\bigg)\vee a_{r}\\
 & =b_{1}\vee\big(a_{1}\wedge(b_{2}\vee a_{2})\wedge\cdots\wedge(b_{r-1}\vee a_{r-1})\wedge b_{r}\big)\vee a_{r}.
\end{align*}
Now a second inductive application of (\ref{eq:BirkhoffAtoms}) yields
the desired.
\end{proof}
\begin{rem}
We replace the two separate arguments of Stanley \cite[Proposition 2.1]{Stanley:1972}
with a simultaneous and self-dual induction, where the $r$th case
of (\ref{eq:BirkhoffCoatoms}) depends on the ($r-1)$st case of both
(\ref{eq:BirkhoffCoatoms}) and (\ref{eq:BirkhoffAtoms}). This argument
seems basically simpler to us than that of Stanley, and we were surprised
not to find it in the literature. We do comment that Thomas gave a
self-dual argument from generally similar hypotheses in \cite[Lemma 9]{Thomas:2005}.
\end{rem}

Since it does not seem to be as well-known as it should be, we briefly
overview the rest of the proof from \cite[pp. 65-66]{Birkhoff:1967};
this proof is also discussed in \cite[Theorem 363]{Gratzer:2011a}.
Let $\mathbf{m}=\left\{ m_{i}\right\} $ be a chain satisfying the
conclusion of Lemma~\ref{lem:ExtBirkhoffLemma}, and $\mathbf{c}=\{c_{j}\}$
be any other chain. Assume without loss of generality that $\mathbf{m}$
and $\mathbf{c}$ both contain $\hat{0}$ and $\hat{1}$. Birkhoff
considers the set $S$ of elements that can be written as a join of
elements of the form $m_{i}\wedge c_{j}$, and the set $S^{*}$ of
elements that can be written as a meet of elements of the form $m_{i}\vee c_{j}$.
He uses (a restricted version of) Lemma~\ref{lem:ExtBirkhoffLemma}
to show that $S=S^{*}$, and so that both coincide with the sublattice
generated by $\mathbf{m}$ and $\mathbf{c}$.

It is easy to see (with no modularity assumption) that if $x=(m_{i_{1}}\wedge c_{j_{1}})\vee\cdots\vee(m_{i_{r}}\wedge c_{j_{r}})$,
then the $i$ indices may be chosen to be increasing and the $j$
indices to be decreasing. It follows after a bit of work that the
sublattice generated by $\mathbf{m}$ and $\mathbf{c}$ is a homomorphic
image of the lattice of order ideals in the direct product $\mathbf{m}\times\mathbf{c}$.
Here, the homomorphism sends the principal order ideal below the pair
$(m_{i},c_{j})$ to $m_{i}\wedge c_{j}$. Thus, the sublattice generated
by $\mathbf{m}$ and $\mathbf{c}$ is the homomorphic image of a lattice
of sets under the operations $\cup$ and $\cap$, hence is distributive.

\section{\label{sec:GradedLM}Equivalence with graded left-modularity}

In this section, we show directly the equivalence of Condition (\ref{enu:ChainModular})
with Conditions (\ref{enu:Char-GradedLM}) and (\ref{enu:RankModular}).

We first prove Lemma~\ref{lem:RankEqLeftModular}, which immediately
gives equivalence of Conditions (\ref{enu:Char-GradedLM}) and (\ref{enu:RankModular}).
We begin by observing:
\begin{lem}
\label{lem:PentagonsBadRankModular}Let $L$ be a graded lattice with
rank function $\rho$. If $z$ together with $x<y$ form a pentagon
sublattice of $L$ (so that $z\wedge x=z\wedge y$ and $z\vee x=z\vee y$),
then we have $\rho(z\vee x)+\rho(z\wedge x)\neq\rho(z)+\rho(x)$ or
$\rho(z\vee y)+\rho(z\wedge y)\neq\rho(z)+\rho(y)$.
\end{lem}

\begin{proof}
The left-hand sides agree, but $\rho(x)<\rho(y)$.
\end{proof}
Lemma~\ref{lem:RankEqLeftModular} follows easily from Lemma~\ref{lem:PentagonsBadRankModular}
together with Lemma~\ref{lem:PentagonsLM}.
\begin{proof}[Proof (of Lemma~\ref{lem:RankEqLeftModular}).]
 If $m$ is not left-modular, then $L$ has a pentagon sublattice
containing $m$, which then fails the rank modular identity by Lemma~\ref{lem:PentagonsBadRankModular}.

Conversely, if $m$ is a left-modular element and $y$ is some other
element of $L$, then consider a maximal chain $x_{0}\lessdot\cdots\lessdot x_{k}$
on $[m\wedge y,y]$. As each $x_{i}$ satisfies $m\wedge x_{i}=m\wedge y$,
repeated application of Lemma~\ref{lem:PentagonsLM} shows that $x_{0}\vee m<x_{1}\vee m<\cdots<x_{k}\vee m$
are distinct elements on $[m,m\vee y]$. Thus, $k=\rho(y)-\rho(m\wedge y)\leq\rho(m\vee y)-\rho(m)$.
Now the same argument on the dual lattice yields that $\rho(y)-\rho(m\wedge y)\geq\rho(m\vee y)-\rho(m)$,
as desired.
\end{proof}
We now show that Condition~(\ref{enu:RankModular}) implies Condition~(\ref{enu:ChainModular})
by a straightforward calculation. If $m_{i}<m_{j}$ are rank modular
elements in the graded lattice $L$ and $z$ is some other element,
then three applications of the rank modular identity yields that 
\begin{align*}
\rho(m_{i}\vee(z\wedge m_{j})) & =\rho(m_{i})+\rho(z\wedge m_{j})-\rho(z\wedge m_{i})\\
 & =\rho(z\wedge m_{j})-\rho(z)+\rho(z\vee m_{i})\\
 & =\rho(m_{j})-\rho(z\vee m_{j})+\rho(z\vee m_{i}).
\end{align*}
A single application of the same identity then yields that $\rho((m_{i}\vee z)\wedge m_{j})=\rho(m_{i}\vee(z\wedge m_{j}))$.
Since the ranks agree, the modular inequality is an equality, as desired.

It remains only to show that Condition (\ref{enu:ChainModular}) implies
Condition (\ref{enu:RankModular}). Let $\mathbf{m}$ consisting of
$\hat{0}=m_{0}\lessdot m_{1}\lessdot\cdots\lessdot m_{n}=\hat{1}$
be a chain-modular maximal chain. Define a function $\rho:L\to\mathbb{N}$
by 
\[
\rho(y)=\#\left\{ i\,:\,m_{i+1}\wedge y>m_{i}\wedge y\right\} .
\]
Since when $x<y$, it holds that $m_{i+1}\wedge x>m_{i}\wedge x$
only if $m_{i+1}\wedge y>m_{i}\wedge y$, the function $\rho$ is
weakly order preserving. We'll show that indeed $\rho$ is the rank
function, and that $\mathbf{m}$ is rank modular.

We first verify the rank modularity condition. It is enough to show
that $\rho(x)$ and $\rho(x\vee m_{j})+\rho(x\wedge m_{j})-\rho(m_{j})$
count the same subset of the modular chain. Indeed, for any $x\in L$,
we notice that:
\begin{itemize}
\item $\rho(x\vee m_{j})=\#\left\{ i\,:\,\left(m_{j}\vee x\right)\wedge m_{i+1}>\left(m_{j}\vee x\right)\wedge m_{i}\right\} $.
Those $i$ with $i<j$ clearly always satisfy the condition. When
$i\geq j$, the condition becomes $m_{j}\vee\left(x\wedge m_{i+1}\right)>m_{j}\vee(x\wedge m_{i})$
by right chain-modularity. Clearly if $x\wedge m_{i+1}=x\wedge m_{i}$
then $i$ does not satisfy the condition (and is not counted). Conversely,
if $x\wedge m_{i+1}>x\wedge m_{i}$, then since $m_{j}\wedge(x\wedge m_{i+1})=m_{j}\wedge(x\wedge m_{i})$,
Lemma~\ref{lem:PentagonsLM} yields that $m_{j}\vee(x\wedge m_{i+1})>m_{j}\vee(x\wedge m_{i})$.\\
In short, $\rho(x\vee m_{j})$ counts the number of $i\geq j$ for
which $m_{i+1}\wedge x>m_{i}\wedge x$, together with all $i<j$.
\item $\rho(x\wedge m_{j})=\#\left\{ i\,:\,\left(m_{j}\wedge x\right)\wedge m_{i+1}>\left(m_{j}\wedge x\right)\wedge m_{i}\right\} $
clearly counts the number of $i<j$ for which $x\wedge m_{i+1}>x\wedge m_{i}$.
\item $\rho(m_{j})=j$ counts the $m_{i}$'s with $i<j$.
\end{itemize}
Thus, among the indices counted by $\rho(x)$, we see that $\left(\rho(x\vee m_{j})-\rho(m_{j})\right)$
counts exactly those indices $i$ with $i\geq j$, while $\rho(x\wedge m_{j})$
counts the indices $i$ with $i<j$.

Now we finish the proof by using the rank modular condition to show
that if $y\gtrdot x$, then also $\rho(y)=\rho(x)+1$. For in this
case, let $i$ be the least index for which $x\wedge m_{i+1}<y\wedge m_{i+1}$.
Then $x\vee m_{i+1}=y\vee m_{i+1}$ by Lemma~\ref{lem:PentagonsLM}.
As $m_{i}\wedge x\wedge m_{i+1}=m_{i}\wedge y\wedge m_{i+1}$ by minimality
of $i$, a second application of Lemma~\ref{lem:PentagonsLM} gives
that $m_{i}\vee x\wedge m_{i+1}<m_{i}\vee y\wedge m_{i+1}$. Thus,
$m_{i}\vee x\wedge m_{i+1}=m_{i}$ and $m_{i}\vee y\wedge m_{i+1}=m_{i+1}$.
We apply rank modularity twice to calculate
\begin{align*}
\rho(y)-\rho(x) & =\rho(y\wedge m_{i+1})-\rho(x\wedge m_{i+1})\\
 & =\rho(m_{i}\vee y\wedge m_{i+1})-\rho(m_{i}\vee x\wedge m_{i+1})\\
 & =\rho(m_{i+1})-\rho(m_{i})=(i+1)-i=1,
\end{align*}
as desired.

\section{\label{sec:CheckingChainMod}How to check chain-modularity}

In this section, we reduce checking right chain-modularity of a maximal
chain $\mathbf{m}$ consisting of left-modular elements to checking
the condition on the cover relations of $\mathbf{m}$.

The main lemma that allows us to do this is as follow:
\begin{lem}
Let $m_{1}<m_{2}<m_{3}$ be elements in a lattice $L$, where $m_{2}$
is left-modular. If $m_{1}<m_{2}$ and $m_{2}<m_{3}$ are both right
chain-modular, then also $m_{1}<m_{2}<m_{3}$ is right chain-modular.
\end{lem}

\begin{proof}
It remains to check that $m_{1}<m_{3}$ satisfies the required right
modularity property. Suppose for contradiction that there is some
$y$ so that $m_{1}\vee(y\wedge m_{3})<(m_{1}\vee y)\wedge m_{3}$.
Then since 
\begin{align*}
m_{1}\vee y\wedge m_{2} & =m_{1}\vee(y\wedge m_{3})\wedge m_{2}<m_{1}\vee(y\wedge m_{3}),\,\text{while}\\
m_{2}\vee y\wedge m_{3} & =m_{2}\vee(m_{1}\vee y)\wedge m_{3}>m_{2}\vee(m_{1}\vee y),
\end{align*}
we have $m_{2}$ incomparable to $m_{1}\vee(y\wedge m_{3})<(m_{1}\vee y)\wedge m_{3}$.
But now these three elements generate a pentagon sublattice, contradicting
left-modularity of $m_{2}$.
\end{proof}
A straightforward inductive argument now yields:
\begin{thm}
Let $L$ be a finite lattice, and $\mathbf{m}$ a maximal chain consisting
of $\hat{0}=m_{0}\lessdot m_{1}\lessdot\cdots\lessdot m_{n}=\hat{1}$.
If each $m_{i}$ is left-modular, and each $m_{i}\lessdot m_{i+1}$
is right chain-modular, then $\mathbf{m}$ is chain-modular.
\end{thm}

Checking whether a cover relation is right chain-modular is easier
than checking an arbitrary chain of two elements. Indeed, in contrast
to right chain-modularity for an arbitrary pair of elements (see Figure~\ref{fig:ModularGeneral}),
we can reduce right chain-modularity of a cover relation to the following
``no pentagon'' condition.
\begin{lem}
If $m_{i}\lessdot m_{i+1}$ is a cover relation in lattice $L$, then
the pair is right chain-modular if and only if $m_{i}\lessdot m_{i+1}$
does not form the long side of any pentagon sublattice of $L$.
\end{lem}

\begin{proof}
For $m<m'$ with $m\vee(x\wedge m')<(m\vee x)\wedge m'$, the latter
two elements together with $x$ form a pentagon sublattice. If $m_{i}\lessdot m_{i+1}$
is a cover relation, then $m_{i}\vee(x\wedge m_{i+1})$ is either
$m_{i}$ or $m_{i+1}$; similarly for $(m_{i}\vee x)\wedge m_{i+1}$.
It follows that if $m_{i}\lessdot m_{i+1}$ fails to be modular, then
$m_{i}\lessdot m_{i+1}$ forms the long side of a pentagon sublattice.
\end{proof}
Thus, Theorem~\ref{thm:MainTheorem} says that a chain $\mathbf{m}$
of a lattice $L$ is a chief chain if and only if no pentagon sublattice
of $L$ has an element of $\mathbf{m}$ as its short side, nor any
cover relation of $\mathbf{m}$ as its long side.

\section{\label{sec:Questions}Questions}

We came to look at finite supersolvable lattices in this light because
of an interest in non-discrete lattices, such as the continuous partition
lattice of \cite{Bjorner:1987,Haiman:1994}. This lattice is \emph{$[0,1]$-graded},
meaning that there is an order-preserving function from $L$ to the
real interval $[0,1]$ that restricts to a bijection on maximal chains.
Moreover, the lattice has a maximal chain whose elements satisfy the
rank modular condition. Which of our characterizations extend from
the finite or discrete case to the $[0,1]$-graded and other non-discrete
cases?
\begin{question}
Is there a description of $[0,1]$-graded lattices having a rank modular
maximal chain in terms of element- and/or chain-wise modularity?
\end{question}

\bibliographystyle{3_Users_russw_Documents_Research_mypapers_A_mod___ization_of_supersolvable_lattices_hamsplain}
\bibliography{2_Users_russw_Documents_Research_Master}

\end{document}